\documentclass[12pt]{article}

\usepackage{amsthm,amsmath,amssymb}

\usepackage{graphicx}
 \usepackage{epsfig}
  \usepackage{latexsym, epsfig, psfrag,eepic,colordvi,bm}
  \usepackage{graphics,color}
  \usepackage{graphics}
  \usepackage{graphicx}
  \usepackage{epsfig}
  \usepackage[all]{xy}

  \usepackage{amssymb}
  \usepackage{amsthm,amsmath}

  \usepackage{latexsym, epsfig, psfrag,eepic,colordvi,bm}
  \usepackage{graphics,color}
  \usepackage{amsmath,amsfonts,amssymb,amscd}
  \usepackage[all]{xy}
  \usepackage{epstopdf}

  \usepackage{amsthm,amsmath,amssymb}

  \usepackage{graphicx}
  \usepackage{amsthm,amsmath,amssymb}

  \usepackage{graphicx,epsfig}
\usepackage[colorlinks=true,citecolor=black,linkcolor=black,urlcolor=blue]{hyperref}

\usepackage[all]{xy}
\usepackage{graphicx}
\usepackage{mathptmx}

\usepackage{latexsym, epsfig,cite, psfrag,eepic,colordvi,bm}
\usepackage{graphics,color,graphicx}
\usepackage[justification=centering]{caption}
\usepackage{amsmath,amsfonts,amssymb,amscd}
\usepackage[all]{xy}

\theoremstyle{plain}
\newtheorem{theorem}{Theorem}[section]

\newtheorem{lemma}{Lemma}[section]

\theoremstyle{definition}

\newtheorem{example}{Example}[section]
\newtheorem{conjecture}[theorem]{Conjecture}

\theoremstyle{remark}

\newcommand{\bai}{\hspace{6pt}}

\date{}

\title{On products of permutations with the most uncontaminated cycles by designated labels}

\author{Ricky X. F. Chen~\footnote{ORCID: 0000-0003-1061-3049}\\
	\small School of Mathematics, Hefei University of Technology \\[-0.8ex]
	\small Hefei, Anhui 230601, P.~R.~China\\[-0.8ex]
	\small\tt xiaofengchen@hfut.edu.cn
}

\begin{document}

\maketitle

\begin{abstract}
	There is a growing interest in studying the distribution of certain labels in products of permutations
	since the work of Stanley addressing a conjecture of B\'{o}na.
	This paper is concerned with a problem in that direction.
	Let $D$ be a permutation on the set $[n]=\{1,2,\ldots, n\}$ and $E\subset [n]$.
	Suppose the maximum possible number of cycles uncontaminated by the $E$-labels in the product of $D$ and a cyclic permutation on $[n]$
	is $\theta$ (depending on $D$ and $E$).		 	 
	We prove that for arbitrary $D$ and $E$ with few exceptions, the number of cyclic permutations $\gamma$ such that $D\circ \gamma$ has exactly
	$\theta-1$ $E$-label free cycles is at least $1/2$ that of $\gamma$ for $D\circ \gamma$ to have $\theta$ $E$-label free cycles,
 where $1/2$ is best possible. An even more general result is also conjectured.

  \bigskip\noindent \textbf{Keywords:}  Factorization of a permutation, Permutation products, Cycle, Bijective function, Top connection coefficient 

  \noindent\small Mathematics Subject Classifications 2020: 05A05, 05E16
\end{abstract}

\section{Introduction}

Let $[n]=\{1,\dots,n\}$, and let $\mathfrak{S}_n$ denote the group of permutations on $[n]$.
A permutation $\pi \in \mathfrak{S}_n$ can be represented by a set of disjoint cycles.
The number of disjoint cycles of $\pi$ is denoted by $|\pi|$, and
the set (with repetition allowed) consisting of the lengths of these disjoint cycles is called the cycle-type of $\pi$.
A cycle of length $k$ is called a $k$-cycle.
A long cycle or cyclic permutation on $[n]$ is an $n$-cycle.

 There are enormous number of works on
studying decomposing a long cycle on $[n]$ into permutations on $[n]$ or products involving a long cycle in various contexts; see for
instance,
Chen and Reidys~\cite{chr1}, Chapuy, F\'{e}ray and Fusy~\cite{cff}, F\'{e}ray and Vassilieva~\cite{fv},
Goulden and Jackson~\cite{IJ2}, Goupil and Schaeffer~\cite{ag}, Jackson~\cite{jac},
Stanley~\cite{stan1}, Walsh and Lehman~\cite{walsh1}, Zagier~\cite{zag}.
In particular, a decomposition of a long cycle into
two permutations determines a one-face hypermap; if one of the said two permutations is
a fixed-point free involution, the decomposition determines a (rooted) one-face map.
The techniques vary widely, from bijective methods to the character theory approaches.

When studying products or factorizations of permutations, often features such as the number of cycles and cycle-type
that are irrelevant to the distribution of certain labels
are the main objects.
For example, the celebrated problem of enumerating rooted one-face maps of distinct genera is merely enumerating
cyclic permutations (up to a factor) whose product with a fixed permutation consisting of only
$2$-cycles yield permutations with distinct numbers of cycles.
Whether certain labels are contained in the same cycle or distinct cycles, or questions alike,
do not matter.  However, there is a growing interest in such questions recently.
The first work on this may be the one by Stanley~\cite{stan3} who aimed at confirming a conjecture
of B\'{o}na. The latter states that the probability for two given labels to be contained in distinct cycles of the product of two random long cycles is $1/2$ when $n$ is odd,
and has a surprising application in genome rearrangement problem concerning block-interchange distance~\cite{bona}. Subsequent studies of tracking label distribution include, for instance, Bernardi, Du, Morales and Stanley~\cite{separ}, B\'{o}na and Pittel~\cite{bona-pittel}, Chen~\cite{chen3},
F\'{e}ray and Rattan~\cite{fr}.
Here we add a new contribution to the subject.

	Let $D$ be a permutation on $[n]$ and $E\subset [n]$.
Suppose the maximum possible number of cycles uncontaminated by the labels from $E$ (i.e., $E$-label free) in the product of $D$ and a cyclic permutation on $[n]$
is $\theta(D,E)$.	
In this paper, our main result is:

\begin{theorem}\label{thm:track}
	For any $D \in \mathfrak{S}_n$ and $E \subset [n]$ with $|E|>1$ such that $D$ has either a $E$-label free cycle of length greater than one or a cycle mixing some labels from $E$ and $[n]\setminus E$, the number of cyclic permutations $\gamma$ such that $D\circ \gamma$ has exactly
	$\theta(D,E)-1$ $E$-label free cycles is at least $1/2$ that of $\gamma$ for $D\circ \gamma$ to have $\theta(D,E)$ $E$-label free cycles, where $1/2$ is best possible.

\end{theorem}

For example, suppose $n=4$ and $D=(1 \bai 2)(3\bai 4)$. The products of $D$ with the six cyclic permutations on $[4]$ are
\begin{align*}
	(1)(2 \bai 4)(3), \, (1)(2\bai 3) (4) ,\, (1\bai 4 \bai 2 \bai 3),\,
	(1\bai 4)(2) (3), \, (1\bai 3 \bai 2 \bai 4), \, (1 \bai 3)(2)(4).
\end{align*}
If $E=\{1,2\}$, then $\theta(D,E)=1$ and there are four products with one $E$-label free cycle while two
with zero $E$-label free cycle;
If $E=\{1,3\}$, then $\theta(D,E)=2$ and there is one product with two $E$-label free cycles while three
with one $E$-label free cycle;
If $E=\{1,2,3\}$, then $\theta(D,E)=1$ and there are two products with one $E$-label free cycle while four
with zero $E$-label free cycle.

In view that $\theta(D,E)$ depends on the relative position of the $E$-lables in the cycles of $D$ and $D$ is an arbitrary permutation while $E$ could be a subset
of labels almost arbitrarily spreading in the cycles of $D$,
the result is quite general.
Our approach here may be of independent interest as well.
We prove a theorem about factorizing a general bijective function (instead of a permutation) at first,
and then briefly explain how to get to Theorem~\ref{thm:track}.


\section{Bijective functions and their factorizations}\label{sec2}
%
%
%

\subsection{Bijective functions}
Let two sets $A=\{a_1,a_2,\ldots, a_n\}$ and $B=\{b_1,b_2,\ldots, b_n\}$. Suppose $f: A\rightarrow B$ is
a bijective function (or bijection). The map $f$ is associated with a directed graph $G_f$ on $A\bigcup B$, where there is a directed edge from $x$ to $y$ if $f(x)=y$. The graph $G_f$ may have multiple disconnected components, and each component is either a directed cycle or a directed path.
In the following, by cycles and paths, we always mean directed cycles and directed paths, respectively.
The lemma below should not be hard to observe.

\begin{lemma}
	 The elements in a cycle must be contained in $A\bigcap B$; in a directed path, the starting end is contained in $A\setminus B$, the terminating end is contained in $B\setminus A$, while the (internal) elements other than the two ends are contained in $A\bigcap B$; each element in $B\setminus A$ must be the terminating end of a directed path so there are exactly $|B\setminus A|=|A\setminus B|$ directed paths in $G_f$.
	\end{lemma}

We use $|f|$ to denote the number of components in $G_f$.
The size of a cycle is the number of elements contained in the cycle; the
size of a path is one less than the number of elements contained in the path, that is
the number of elements in $A$ which are contained in the path.
The sizes of these disconnected components in $G_f$ determine the component-type of $f$ which
can be encoded into an ordered pair of non-increasing integer sequences $(\lambda, \mu)=(\lambda_1\lambda_2\ldots, \mu_1\mu_2\ldots)$, where $\sum_i \lambda_i +\sum_j \mu_j=n$,
meaning that there is exactly one path of size $\lambda_i$, and one cycle of size $\mu_j$ for all possible $i,j$.
We write $(\lambda, \mu) \models n$.
This concept of component-type is analogous to the cycle-type of a permutation
in symmetric groups.

\begin{example}
	Let $A=\{1,2,3,4\}$ and $B=\{2,3,5,6\}$. A map $f:A\rightarrow B$ is specified as follows:
	$$
	f(1)=6, \quad f(2)=3,\quad  f(3)=2, \quad f(4)=5.
	$$
	Then, the component-type of $f$ is $(11,2)$, as there are two directed paths $1 \rightarrow 6$
	and $4 \rightarrow 5$, as well as one directed cycle $(2 \bai 3)$.
	
\end{example}

For two sets $A$ and $B$ with $|A|=|B|=n$,
decomposing a bijection $g$
from $A$ to $B$ into a long cycle $s$ on $A$ and a bijection $f$ from $A$ to $B$, we arrive at
$g=f\circ s$.
This obviously includes decomposing a long cycle into permutations as a special case
and deserves a systematic study.

In a prior work, the author~\cite{chen1} proved an interpolation theorem on the attainable number
of components (i.e., directed cycles and directed paths) in $f$. Notably,
these attainable numbers do not have to have the same parity as that of permutations and the maximum number of components $f$
may have was proved to be $n-|g|+|A\setminus B|$ if $|A\setminus B|>1$, where $|g|$ denotes the number of components in $g$.
In contrast, the permutation case, that is, $|A\setminus B|=0$, the maximum
is known to be $n+1-|g|$.

The number of decompositions of a permutation $g=f \circ s$ where $f$ has $n+1-|g|$
components (actually cycles) is referred to as a top connection coefficient and
has been studied for instance in Goulden and Jackson~\cite{IJ1,IJ2} (see also B\'{e}dard and Goupil~\cite{bg92}).
When $g$ is a fixed point free involution, it is well known that
the corresponding top coefficient is essentially the famous Catalan number.
An analog in complex reflection groups, i.e., genus-$0$ connection coefficients,
has also been investigated by Lewis and Morales~\cite{lm21} recently.

Our problem here is related to the top decompositions of a general bijective function,
and
the forthcoming framework of two-row arrays will be helpful in many ways.


\subsection{Two-row arrays: a tool to study factorizations}

Let $\varPsi$ be a two-row array:
$$
\varPsi=\left(\begin{array}{ccccc}
s_0&s_1&\cdots &s_{n-2}&s_{n-1}\\
f(s_0)&f(s_1)&\cdots &f(s_{n-2}) &f(s_{n-1})
\end{array}\right),
$$
where the sequence $s_0s_1\cdots s_{n-1}$ is a permutation on $A$, while $f(s_0)f(s_1)\cdots f(s_{n-1})$ is a permutation on $B$. We say the two-row array $\varPsi$ is on $A$ and $B$, denoted by $[\varPsi: A\mid B]$. Let $s$ denote the cyclic permutation $(s_0\bai s_1\bai\cdots \bai s_{n-1})$ on $A$, and let $f$ be the (vertical) map
$$
f: A\rightarrow B,\bai s_i\mapsto f(s_i), \bai 0\leq i \leq n-1.
$$
In addition, let $D_{\varPsi}$ be the map
$$
D_{\varPsi}: B\rightarrow A, \bai f(s_i)\mapsto s_{i+1}, \bai 0\leq i \leq n-1,
$$
where $s_n$ is treated as $s_0$. The map $D_{\varPsi}$ is called the diagonal of $\varPsi$. Then, the following lemma is clear but crutial.

\begin{lemma}\label{lem:triple}
The upper horizontal $s$, the vertical and the diagonal satisfy
	$
	s=D_{\varPsi}\circ f.
	$
	\end{lemma}


We also employ $\varPsi=(s,f)$ as a shorthand for the explicit two-row array representation if there is no
confusion about which is the left-most element $s_0\in A$.
The idea of interpreting two-row arrays in such a way follows from plane permutations~\cite{chen3,chr1}.

Let $\mathfrak{T}_k(D,x)$ be the set of two-row arrays $\varPsi=(s,f)$ on the set $A$ and $B$ such that the left-most
element on the upper horizontal is the fixed element $x\in A$, the diagonal is the fixed one-to-one map $D:B\rightarrow A$ and there are $k$ components in the vertical map.
Based on Lemma~\ref{lem:triple}, the set $\mathfrak{T}_k(D,x)$ corresponds to the set of ways of decomposing
the fixed map $D$ into a long cycle on $A$ (i.e., $s$) and a map from $B$ to $A$ (i.e., $f^{-1}$) with
$k$ components, i.e., $D=s\circ f^{-1}$.
For this reason, the two concepts, two-row arrays and decompositions (essentially triples of
bijections) will be used interchangeably, whichever is more convenient.

By a relabelling argument, it is clear to see that
$|\mathfrak{T}_k(D_1,x)|=|\mathfrak{T}_k(D_2,y)|$
as long as $D_1, D_2: B\rightarrow A$ have the same component-type $(\lambda,\mu)$.
For this reason, we define the number
$$
W_k^{\lambda,\mu}:=|\mathfrak{T}_k(D_1,x)|.
$$

Given $\varPsi=(s,f)$ and a sequence $h=(i,j,k)$, such that $i\leq j<k$
and $\{i,j,k\}\subset [n-1]$, if we transpose the two diagonal-blocks determined by the continuous segments $[s_i,s_j]$ and $[s_{j+1},s_k]$, we obtain a new two-row array $\varPsi^h=(s^h,f^h)$:

\begin{eqnarray*}
\left(
\vcenter{\xymatrix@C=0pc@R=1pc{
\cdots & s_{i-1}  & s_{j+1}\ar@{--}[dl] &\cdots & s_{k-1}&\bai s_k\ar@{--}[dl] \bai\bai& s_i\ar@{-}[dl] &\cdots & s_{j-1} & s_{j}\ar@{-}[dl] & s_{k+1}  &\cdots\\
\cdots  & f(s_{j}) & f(s_{j+1}) & \cdots & f(s_{k-1})& f(s_{i-1}) & f(s_i) & \cdots\hspace{-0.5ex} & f(s_{j-1}) & f(s_{k})  & f(s_{k+1})&\cdots
}}
\right).
\end{eqnarray*}
Obviously, $\varPsi$ and $\varPsi^h$ have the same diagonal; and the maps $f$ and $f^h$ only differ at the images of the elements $s_{i-1}$, $s_j$, and $s_k$.
The latter implies that all components other than those containing the mentioned three elements of $f$ (i.e., $G_f$) will be completely carried over to $f^h$ without any changes.
For those components containing the three elements, the three elements to some extent serve as breakpoints, where
the induced segments will be re-pasted in a certain way, depending on the distribution of the elements $s_{i-1}$, $s_j$, and $s_k$ in the components of $f$.
The following two particular cases will be used later.

\begin{lemma}\label{lem:3-to-2}
If in $f$, the elements $s_{i-1}$, $s_j$, and $s_k$ are contained in two paths and one cycle, then after the diagonal transposition, the three components will become two components.
\end{lemma}
\proof We just show one case, the rest of cases can be checked analogously.
Assume the three components are:
$$
\cdots_{il} \rightarrow s_{i-1}\rightarrow \cdots_{ir}, \qquad \cdots_{jl} \rightarrow s_{j}\rightarrow \cdots_{jr}, \qquad (s_{k}\bai \cdots_k).
$$
Then after the transposition, they will reorganize into the following two components:
$$
\cdots_{il}\rightarrow s_{i-1}\rightarrow \cdots_{jr}, \quad \cdots_{jl}\rightarrow s_j\rightarrow \cdots_k \rightarrow s_k \rightarrow \cdots_{ir},
$$
completing the proof. \qed

Analogously, we obtain the following lemma whose proof is left to the interested reader.
\begin{lemma}\label{lem:split-path}
If in $f$, the elements $s_{i-1}$, $s_j$, and $s_k$ are in the same path of the form
	$$
	\cdots_i \rightarrow s_{i-1}\rightarrow v_1^i \rightarrow \cdots\rightarrow v_{m_i}^i \rightarrow s_k \rightarrow v_1^k \rightarrow \cdots \rightarrow v_{m_k}^k \rightarrow s_j  \rightarrow v_1^j \rightarrow \cdots \rightarrow v_{m_j}^j\rightarrow b_t,
	$$
then in $f^h$, the path splits into the following three components:
	$$
		\cdots_i \rightarrow s_{i-1}\rightarrow v_1^j \rightarrow \cdots \rightarrow v_{m_j}^j\rightarrow b_t, \qquad  (s_j \bai v_1^k \bai \ldots \bai v_{m_k}^k),\qquad  (s_k \bai v_1^i \bai\ldots \bai v_{m_i}^i).
	$$
\end{lemma}

\section{Proof of Theorem~\ref{thm:track}}

Now we are in a position to present an equivalent result of Theorem~\ref{thm:track}.

\begin{theorem}\label{thm:half}
For any $(\lambda,\mu)\models n$ with $2\leq \ell(\lambda)<n-\ell(\mu)$, we have
\begin{align}
W_{n-\ell(\mu)-1}^{\lambda,\mu}\geq \frac{1}{2} W_{n-\ell(\mu)}^{\lambda,\mu},
\end{align}
where $1/2$ is best possible.
\end{theorem}
\proof
Let $E=\{a_1,a_2,\ldots, a_{\ell(\lambda)}\}\subset [n]$ and $V=\{b_1,b_2,\ldots,b_{\ell(\lambda)}\}$ with $b_i \notin [n]$
for $1\leq i \leq \ell(\lambda)$,
and let $A=[n]$ and $B=\left([n]\setminus E\right)\bigcup V$.
Suppose the component-type of $D:B\rightarrow A$ is $(\lambda,\mu)$.
We prove the theorem by constructing a map $\phi$ from the set $\mathfrak{T}_{n-\ell(\mu)}(D,1)$
to the set $\mathfrak{T}_{n-\ell(\mu)-1}(D,1)$, where any $\varPsi$ in the latter set has at most $2$ preimages.

Without loss of generality, assume $a_1=1$ and $a_2=2$. For a $\varPsi=(s,f)\in \mathfrak{T}_{n-\ell(\mu)}(D,1)$,
there exists at least one cycle since $n-\ell(\mu)>\ell(\lambda)$.
Pick the paths containing $1$ and $2$, and the cycle
containing the minimum $m$ over all of the elements
contained in the cycles of $f$.
Assume these three components are of the form:
$$
1\rightarrow \cdots_1, \qquad 2\rightarrow \cdots_2, \qquad (m\bai \cdots_m).
$$
According to Lemma~\ref{lem:3-to-2}, the transposition action determined
by the three elements $1$, $2$ and $m$ will lead to an element $\phi(\varPsi)=(s',f')\in \mathfrak{T}_{n-\ell(\mu)-1}(D,1)$.
There are the following two cases:
\begin{itemize}
\item[(i).] If $1<_s 2 <_s m$, then after the transposition, we obtain the two components:
$$
1\rightarrow \cdots_2, \qquad 2 \rightarrow \cdots_m \rightarrow m \rightarrow \cdots_1.
$$
Following the new order $<_{s'}$, we have
$1<_{s'} m<_{s'} 2$. Furthermore, since $m$ is the minimum in all cycles before the
transposition, $2$ is smaller than all these elements contained in the cycles but $m$ w.r.t.~$<_{s'}$. Hence,
in the second component above, $m$ is the first element to the
right of $2$ which is smaller than $2$, and all elements contained in
the cycles of $f'$ are greater than $2$.
\item[(ii).] If $1<_s m <_s 2$, after the transposition, we obtain
$$
1\rightarrow \cdots_m \rightarrow m \rightarrow \cdots_2, \qquad 2 \rightarrow \cdots_1.
$$
Following the new order $<_{s'}$, we have
$1<_{s'} 2<_{s'} m$. Furthermore, $m$ is an element to the
right of $1$ which is greater than $2$, the elements in the first component between $1$ and $m$ are respectively
either smaller than $2$ or greater than $m$, and all elements contained in
the cycles of $f'$ are respectively either smaller than $2$ or greater than $m$.
\end{itemize}

Conversely, for a given element $\varPsi'=(s',f')\in \mathfrak{T}_{n-\ell(\mu)-1}(D,1)$, we can try to find its preimages under $\phi$ as follows:
\begin{description}
\item[Step~$1$.] If there are elements contained in the cycles of $f'$ which are smaller than $2$, $\varPsi'$ certainly
does not have a preimage of the case (i). It remains to check
if we can find the element $m$ in the case (ii). If not, then $\varPsi'$ has no preimage under $\phi$.
\item[Step~$2$.] If all elements in the cycles of $f'$ are greater than $2$, there are three
situations: (1) we can neither find $m$ satisfying the case (i) nor the case (ii); (2)
we can only find $m$ satisfying one of them; (3) we can find $m$ satisfying both of them, respectively.
\end{description}

\emph{Claim~$1$.} If $\varPsi'$ has a preimage from the case (i), the preimage is unique w.r.t. the case (i).\\
Note that the desired $m$ should be
the first element to the
right of $2$ that is smaller than $2$. Obviously, the first such element is unique, whence Claim~$1$.

\emph{Claim~$2$.} If $\varPsi'$ has a preimage from the case (ii), the preimage is unique w.r.t. the case (ii).\\
Suppose otherwise $m_1$ is the first element to the right of $1$ which lead to a preimage of the case (ii) and
$m_2$ is the second element which lead to a preimage of the case (ii). Obviously, the two preimages must be
different.
Since the elements between $1$ and $m_2$ are respectively either smaller than $2$ or
greater than $m_2$, we have $m_1>_{s'}m_2$. Thus, there exists $x$ between $m_1$ (included) and $m_2$ (excluded) on the path such
that $x>_{s'} f'(x)$. Then, in view of Lemma~\ref{lem:split-path},
the transposition determined by the three elements $1$, $x$ and $f'(x)$ will give
a two-row array where the vertical has $2$ more components. This contradict the fact
that the current two-row array has one less component than the maximum achievable (i.e., $n-\ell(\mu)$) in the vertical.
Hence, Claim~$2$ follows.

As a consequence, each element in $\mathfrak{T}_{n-\ell(\mu)-1}(D,1)$ has at most two preimages,
completing the proof of the inequality.
As for the sharpness, we refer to the following example.
Let $A=\{1,2,3,4\}$, $B=\{3,4,5,6\}$, and the components of $D$ are $6\rightarrow 1$,
$5 \rightarrow 2$ and $(3\bai 4)$.
Then, the two-row arrays in $\mathfrak{T}_{3}(D,1)$ are
\begin{align*}
\left(\begin{array}{cccc}
				1 & 2 & 3 &4\\
				5 & 4 & 3 &6
				\end{array}\right),\,
\left(\begin{array}{cccc}
				1 & 2 & 4 &3\\
				5 & 3 & 4 &6
				\end{array}\right),\,
\left(\begin{array}{cccc}
				1 & 3 & 4 &2\\
				4 & 3 & 5 &6
				\end{array}\right),\,
\left(\begin{array}{cccc}
				1 & 4 & 3 &2\\
				3 & 4 & 5 &6
				\end{array}\right),
\end{align*}
and there are exactly two two-row arrays in $\mathfrak{T}_{2}(D,1)$:
\begin{align*}
\left(\begin{array}{cccc}
				1 & 3 & 2 &4\\
				4 & 5 & 3 &6
				\end{array}\right),\,
\left(\begin{array}{cccc}
				1 & 4 & 2 &3\\
				3 & 5 & 4 &6
				\end{array}\right).
\end{align*}
This completes the proof.
\qed

\begin{proof}[Proof of Thorem~\ref{thm:track}]
	Suppose $E=\{a_1,a_2,\ldots, a_l\}$,
	and a long cycle $s=D\circ \pi$ where $\pi$ has $j$ cycles not containing any element from $E$.
	Then, clearly the two-row array $(s,\pi) \in \mathfrak{T}_{|\pi|}(D,1)$.
	We next replace $a_i$ in the second row of the array with $b_i \notin [n]$ for $1\leq i \leq l$
	and denote by $(s,\pi')$ the resulting array.
	It is easily seen that this replacement does not affect the cycles of $\pi$ and $D$ that
	containing no $E$-labels, and these $b_i$'s will break cycles into paths.
	Let $B=([n]\setminus E)\bigcup \{b_1,\ldots, b_l\}$.
	Consequently, $\pi'$ is a bijective map from $[n]$ to $B$ that 
	has $l$ paths and $j$ ($E$-label free) cycles, and
the diagonal $D'$ of $(s,\pi')$ is a bijective map from $B$ to $[n]$ that
consists of $l$ paths and a certain number of cycles.
Suppose $D'$ is of component-type $(\lambda, \mu)$.
	Then, it may be not hard to see that $D$ having a $E$-label free cycle of length greater than one or a cycle mixing both kinds of labels implies $2\leq \ell(\lambda)< n- \ell(\mu)$.
	Conversely, given any $(s, \pi') \in \mathfrak{T}_{k}(D',1)$ (for some $k$) where $\pi'$ has $j$ cycles, if we replace $b_i$ with $a_i$, then
	we obtain $s=D\circ \pi$ with $\pi$ having $j$ $E$-label free cycles.
	This correspondence is obviously one-to-one, and the rest of the proof should be easy to complete. 
\end{proof}

We remark that Theorem~\ref{thm:half} is not true for $\ell(\lambda)=0$ and $\ell(\lambda)=1$.
Based on some data, we also make the following general conjecture which of course has an equivalent form in the spirit of Theorem~\ref{thm:track}.
\begin{conjecture}\label{conj:1}
There exists a fixed constant $0<\delta \leq 1/2$ such that for any $(\lambda,\mu) \models n$ ($n>6$) with $\ell(\lambda)\geq 2$ and $k\geq \ell(\lambda)$,
$
W_k^{\lambda,\mu}\geq \delta \, W_{k+1}^{\lambda,\mu}.
$
\end{conjecture}

\subsection*{Statements and Declarations}

{\bf Competing Interests:} The author declares no competing interest.


\end{document}